\documentclass[12pt,reqno, a4paper]{amsart}
\usepackage[english]{babel}
\usepackage{amsthm}
\usepackage[T1]{fontenc}
\usepackage[utf8]{inputenc}
\usepackage[english]{babel}
\usepackage{graphicx, color}
\usepackage{amssymb}
\usepackage{amsmath, xfrac}
\usepackage{latexsym}
\usepackage{caption, subcaption}
\usepackage{lineno}
\usepackage{comment}
\usepackage{mathtools}  
\usepackage{enumitem} 
\usepackage{hyperref}
\usepackage{cite}
\usepackage{kpfonts}
\usepackage{multirow}
\usepackage[foot]{amsaddr}
\usepackage{float}

\newcounter{contSect} \numberwithin{contSect}{section}
 \numberwithin{contSub}{subsection}

\newtheorem{theorem}[contSect]{Theorem}

\newtheorem{lemma}[contSect]{Lemma}
\newtheorem{proposition}[contSect]{Proposition}

\newtheorem{claim}[contSect]{Claim}

\newtheorem{problem}{Problem}

\usepackage[letterpaper,top=2cm,bottom=2cm,left=2.5cm,right=2.5cm,marginparwidth=1.75cm]{geometry}

\title{There are many 5-holes}
\author{Omar Astudillo-Marbán$^1$ \and Oriol Solé-Pi $^2$}
\address{$^1$ Massachusetts Institute of Technology. ofam01@mit.edu}
\address{$^2$ Department of Mathematics, Massachusetts Insitute of Technology. oriolsp@mit.edu}

\date{}

\begin{document}

\begin{abstract}
Given a set $P$ of points on the plane, a polygon with vertices in $P$ is said to be \emph{empty} if it contains no element of $P$ in its interior. We show that every set of $n$ points in general position on the plane determines at least $\Omega(n^{20/11})$ empty convex pentagons (also known as $5$-holes). This result improves upon the previous bound of $\Omega(n\cdot(\log n)^{4/5})$ obtained by Aicholzer et al.~\cite{aichholzer2020superlinear}, and significantly narrows the gap with respect to the conjectured $\Omega(n^2)$ lower bound (which, if true, would be tight). Unlike some of the other works in this line of research, our proof does not require computer assistance.
\end{abstract}
\small\maketitle

\section{Introduction}
Consider a set $P$ of points on the plane with no three elements lying on the same line (i.e., $P$ is in \textit{general position}). A triangle with vertices in $P$ is said to be a $3$\textit{-hole} if there is no point of $P$ in its interior. More generally, by a $k$\textit-hole of $P$ we mean a convex $k$-sided polygon with vertices in $P$ which contains no element of $P$ in its interior. We will also refer to convex polygons which contain no points of $P$ in their interior as being \textit{empty}.

The study of $k$-holes was initiated in the 1970s by Erd\H{o}s~\cite{erdos1978some}, who asked whether for every positive integer $k$ there exists a $k$-hole in every sufficiently large set of points in general position on the plane. It is easy to see that any set containing at least $5$ points must induce a $4$-hole, and Harborth~\cite{harborth1978konvexe} showed that $10$ points always determine a $5$-hole---both of these results are optimal in regards to the number of points required. A few years later, Horton~\cite{horton1983sets} managed to resolve Erd\H{o}s' question by providing, for every positive integer $n$, a set of $n$ points in general position which determines no $k$-hole for any $k\geq 7$; the point sets he constructed in this paper have come to be known as \textit{Horton sets}. For $k=6$, the question turned out to be significantly harder to address, and it was not until 2007 that Nicolás~\cite{nicolas2007empty} and Gerken~\cite{gerken2008empty} (independently) proved that every sufficiently large set of points must contain a $6$-hole. In the other direction, Overmars~\cite{overmars2002finding} found a set of $29$ points containing no $6$-hole using a local search approach. Recently, a matching upper bound was found by Heule and Scheucher~\cite{heule2024happy}, who utilized novel computational techniques to prove that every set of $30$ points contains a $6$-hole. 

In the study of $k$-holes, there is another closely related question that has received a lot of attention over the last 35 or so years: what is the least possible number of $k$-holes that a set of $n$ points in general position can determine? From now on, we denote this quantity by $h_k(n)$. According to later works, Purdy~\cite{purdy1982minimum} appears to have been the first to consider this question (at least for $k=3$), although we have been unable to find the corresponding conference proceedings. A more systematic study of this problem was initiated in 1987-1988 by Katchalski and Meir~\cite{katchalski1988empty}, Bárány and Füredi~\cite{barany1987empty}, and Dehnhardt~\cite{Dehnhardt1987}. In particular, it was shown in~\cite{katchalski1988empty} that $h_3(n)\leq 200n^2$, and the following lower bounds were obtained in~\cite{barany1987empty}:
\begin{itemize}
    \item $ h_3(n)\geq n^2-O(n\log n)\,;$

    \item $h_4(n)\geq \frac{1}{2}n^2-O(n\log n)\,;$

    \item $h_5(n)\geq \lfloor\frac{n}{10}\rfloor\,.$
\end{itemize}
Note that the last inequality follows easily from the fact that any $10$ points determine a $5$-hole. In fact, Dehnhardt~\cite{Dehnhardt1987} obtained the stronger bounds $h_3(n)\geq n^2-5n+10$ and $h_5(n)\geq3\lfloor\frac{n}{12}\rfloor$, although these result went largely unnoticed for many years after their publication. By exploiting the aforementioned construction of Horton~\cite{horton1983sets}, Bárány and Füredi~\cite{barany1987empty} also managed to show that, for $n$ a power of two, \begin{itemize}
    \item $ h_3(n)\leq2 n^2\,;$

    \item $h_4(n)\leq 3n^2\,;$

    \item $h_5(n)\leq 2n^2\,;$

    \item $h_6(n)\leq \frac{1}{2}n^2\,.$
\end{itemize}

Subsequent improvements by Valtr~\cite{Valtr1995}, Dumitrescu~\cite{dumitrescu2000planar}, and finally Bárány and Valtr~\cite{barany2004planar}, have led to the currently best known upper bounds for these quantities:
\begin{itemize}
    \item $h_3(n)\leq 1.6196n^2 + o(n^2)\,;$

    \item $h_4(n)\leq 1.9397n^2 + o(n^2)\,;$

    \item $h_5(n)\leq 1.0207n^2+o(n^2)\,;$

    \item $h_6(n) \leq 0.2006n^2 + o(n^2)\,.$
\end{itemize}

Regarding lower bounds, Bárány and Károlyi~\cite{barany2000problems} asked the following in the year 2000:

\begin{problem}[Bárány and Károlyi~\cite{barany2000problems}]
    Is there an absolute constant $\epsilon >0$ such that $h_3(n)\geq (1+\epsilon)n^2$ for every sufficiently large integer $n$?
\end{problem}

It is widely believed that the answer to this question is positive. As was shown in~\cite{pinchasi2004empty}, a proof of this would also imply that the answer to both of the following questions is yes.

\begin{problem}
    Is there a constant $\epsilon_1>0$ such that $h_4(n)\geq \left(\frac{1}{2}+\epsilon_1\right)n^2$ for every sufficiently large integer $n$? Is there a constant $\epsilon_2>0$ with the property that $h_5(n)\geq \epsilon_2 n^2$ for all sufficiently large $n$?
\end{problem}

The interest in these question has led several people on a quest to improve the lower bounds on the number of $k$-holes for $k\in\{3,4,5\}$. The focus of our work is on the lower bound for $h_5(n)$. In this direction, a long series of improvements by García~\cite{garcia2011note,Garcia2012}, Aichholzer, Hackl, and Vogtenhuber~\cite{aichholzer20115}, Valtr~\cite{10.5555/2523693.2523699} and Aicholzer et al.~\cite{AICHHOLZER2014605} eventually resulted in the bound $h_5(n)\geq \frac{3}{4}n-o(n)$. These improvements were all attained by exploiting the known values of $h_5(n)$ for small values of $n$ (in particular, for $n=10,11,12$). Table~\ref{table:5-holes} reflects the state of the art in regards to the exact value of $h_5(n)$ for small $n$. It was not until 2017 that Aicholzer et al.~\cite{aichholzer2020superlinear} managed to derive a super-linear lower bound for $h_5(n)$. The proof of this result is computer-assisted. 

\begin{theorem}[Aicholzer et al.~\cite{aichholzer2020superlinear}]\label{thm:aich}
    There is an absolute constant $c>0$ such that for all $n\geq 10$ we have \[h_5(n)\geq cn(\log n)^{\frac{4}{5}}\,.\]
\end{theorem}

Until now, this had remained the best known lower bound for $h_n(5)$.

\begin{table}[h]
\begin{tabular}{rrrrrrrrrrrrr}
\multicolumn{1}{l|}{$n$}       & 9 & 10 & 11 & 12 & 13 & 14 & 15 & 16 & 17       & 18        & 19        & 20        \\ \hline
\multicolumn{1}{l|}{ $h_5(n)$ } & 0 & 1  & 2  & 3  & 3  & 6  & 9  & 11 & $\leq16$ & $\leq 21$ & $\leq 26$ & $\leq 33$
\end{tabular}
\caption{The value of $h_n(n)$ for all $n\leq 20$. See~\cite{scheucher2013counting} and~\cite{aichholzer2020superlinear} for further information regarding these values and how they have been obtained.}
\label{table:5-holes}
\end{table}

\subsection{Our result and future work}

As our main contribution, we present a substantial improvement on the lower bound for the number of $5$-holes determined by $n$ points in general position.

\begin{theorem}\label{thm:main}
    There is an absolute constant $c>0$ such that for all $n\geq10$ we have that \[h_5(n)\geq cn^{20/11}\,.\]
\end{theorem}

\noindent\textit{Remark.} We make no effort to optimize $c$. In fact, we never provide an explicitly value of $c$ for which the above holds, although it would be perfectly possible to do so using our arguments. 

Our proof is not computer assisted. In fact, the only input we require is the fact that any $10$ points induce a $5$-hole (which was proven by hand in~\cite{harborth1978konvexe}). For every sufficiently large integer $K$, our techniques also reveal some strong structural properties that a set of $n$ points inducing at most $n^2/K$ $5$-holes would need to satisfy. As one might expect, these structural requirements become more restrictive as $K$ increases from $\Omega(1)$ all the way up to $Cn^{2/11}$. It is possible that our methods could be pushed further, and maybe even lead to a proof of the sought-after $\Omega(n^2)$ lower bound for the number of $5$-holes.

It would also be of interest to try and extract from our arguments some information about the number of $3$ and $4$ holes (perhaps in the style of~\cite{Garcia2012}), or to try to adapt them to attack other related problems, such as the ones studied in the papers~\cite{barany2000problems,pach2013monochromatic,aichholzer2010large,diaz2021note,bhattacharya2026number} and the references therein.

Finally, an easy consequence of the fact that a sufficiently large point set must contain a $6$-hole is that $h_6(n)\geq\Omega(n)$. No superlinear lower bound for $h_6(n)$ is known and, as mentioned above, the best known upper bound for this quantity grows quadratically in $n$. While it seems unlikely that there is much to be said about $h_6(n)$ without first obtaining some further information regarding the appearance of $6$-holes in small point sets, we hope that the present paper will motivate people to work further on trying to understand the growth rate of this function.

\subsection{Preliminaries and notation}
Throughout this work, points on the plane will be denoted by lower case letters, and sets of points will be denoted by upper case letters. Given any two distinct points $a$ and $b$ on the plane, we use $\overleftrightarrow{ab}$ to indicate the line through them, oriented from $a$ to $b$, and we write $\overrightarrow {ab}$ to denote the infinite ray starting at $a$ and going through $b$. We also let $\overline{ab}$ be the segment with endpoints $a$ and $b$. For any three pairwise distinct points $a$, $b$ and $c$ on the plane, we define $\angle abc$ as the closed angular region contained in-between the rays $\overrightarrow{ba}$ and $\overrightarrow{bc}$ when moving in clockwise direction around $b$. We shall use $\measuredangle abc$ to indicate the angular measure of $\angle abc$, expressed as a number in $[0,2\pi)$ (see Figure~\ref{fig:angle_definition}).

\begin{figure}[h]
    \centering
    \includegraphics[width=.45\linewidth]{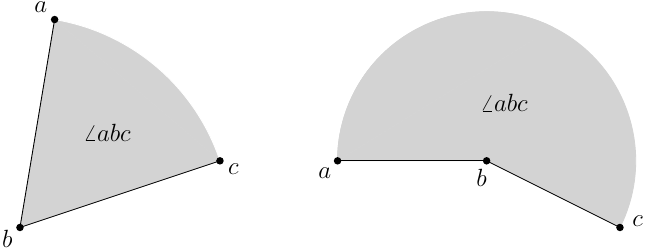}
    \caption{The image shows two possible configurations of three distinct points $a$, $b$ and $c$, along with the angle $\angle abc$ determined by them. On the left, we have that $\measuredangle abc<\pi$, while on the right $\measuredangle abc>\pi$.}
    \label{fig:angle_definition}
\end{figure}

A $k$-vertex polygon with vertices $p_1,\dots,p_k$ (in that order) will be denoted simply by $p_1p_2\dots p_k$. While we will often list the vertices in clockwise order, we make no distinction between orientations here, so the same polygon could also be written as $p_kp_{k-1}\dots p_1$. For clarity, we will add the symbol $\triangle$ in front when talking about triangles---this way, the triangle with vertices $a$, $b$ and $c$ will be denoted by $\triangle abc$. All polygons we consider are \textit{simple} (i.e., only adjacent sides intersect, and they do so only at their common endpoints). Furthermore, we think of polygons as being \textit{solid} (meaning that the finite region enclosed by the boundary is part of the polygon itself).

Given a set of points $P$ in general position on the plane, its \textit{convex hull} is the smallest convex polygon which contains all of its elements. Note that each vertex of the convex hull of $P$ must itself be an element of $P$. We say that $P$ is in \textit{convex position} if all of its elements lie on the boundary of its convex hull (or, equivalently, if all its elements are vertices of the convex hull). We now define the \textit{layer decomposition} of $P$. This decomposition will be an ordered partition of $P$ into sets, called \textit{layers}, each of which is in convex position. The \textit{first layer} $\ell_1$ of the decomposition consists of those elements of $P$ which belong to the boundary of its convex hull. We define the rest of the decomposition inductively: if $\ell_1,\dots \ell_k$ are the first $k$ layers of $P$, then the \textit{k+1}-th layer $\ell_{k+1}$ of $P$ is defined as the first layer of the set $P\backslash (\ell_1\cup\dots \ell_k)$; if $P=\ell_1\cup\dots\cup \ell_k$ holds, then we stop the process. The \textit{number of layers} of $P$ is simply the number layers which conform its layer decomposition, and will be denoted by $L(P)$. See Figure~\ref{fig:layer_decomposition}.

\begin{figure}[h]
    \centering
    \includegraphics[width=.225\linewidth]{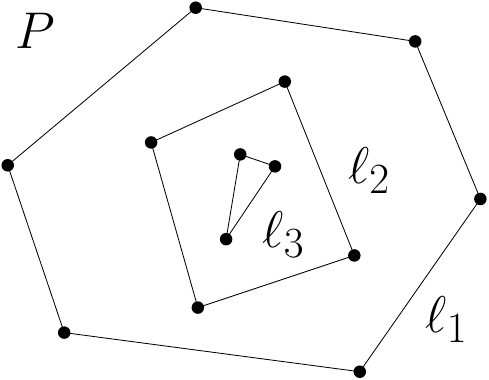}
    \caption{The picture shows a set $P$ of $13$ points with $L(P)=3$. The outer-most layer is $\ell_1$, the inner-most layer is $\ell_3$, and the middle layer is $\ell_2$.}
    \label{fig:layer_decomposition}
\end{figure}

Let $P$ be a set of $n$ points in general position on the plane. Given a point $p\in P$ on the boundary of the convex hull of $P$, the \textit{radial order} of $P$ with respect to $p$ is the unique ordering $p_1,\dots,p_{n-1}$ of the $n-1$ elements of $P\backslash\{p\}$ which satisfies the two following properties: 
\begin{itemize}
\item $\measuredangle p_1pp_{n-1}<\pi$;

\item $p_j\in\angle p_ipp_k$ for every three indices $i,j,k$ with $1\leq i<j<k\leq n-1$.
\end{itemize}
See Figure~\ref{fig:radial_ordering}.

\begin{figure}[h]
    \centering
    \includegraphics[width=.225\linewidth]{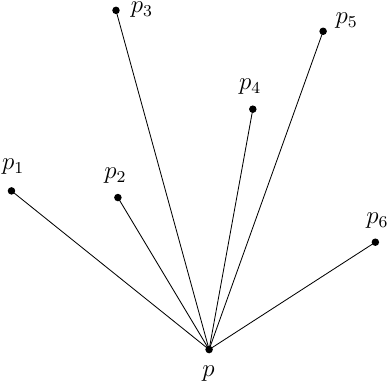}
    \caption{The picture shows a set of $7$ points on the plane. The point $p$ lies on the boundary of its convex hull, and the other $6$ points have been labeled according to the radial order with respect to $p$.}
    \label{fig:radial_ordering}
\end{figure}

Regarding asymptotic notation, throughout the paper we will use the two following definitions of \textit{big-}$O$ and $\Omega$ notation: Given two non-negative functions $f(x)$ and $g(x)$ defined on some unbounded subset of $\mathbb R_{\geq 0}$, we write
\begin{itemize}
    \item $f(x)\leq O(g(x))$ if there is a constant $C>0$ such that $f(x)\leq Cg(x)+C$ holds for all $x$ in the domain of both functions;
    \item and $f(x)\geq\Omega(g(x))$ if there is a constant $c>0$ such that $f(x)\geq cg(x)-c^{-1}$ for all $x$ in the domain of both functions.
\end{itemize}

\subsection*{Organization of the paper}
In Section~\ref{sec:layers}, we will provide a result connecting the number of $5$-holes in $P$ to the number of layers $L(P)$. The proof of this result already showcases one of the main novel ingredients in our approach, namely, a certain way to assign $5$-holes to points. In Section~\ref{sec:technical_lemmas}, we prove several further properties of this assignment of $5$-holes and present a short proof of the inequality $h_5(n)\geq\Omega(n^{4/3})$ (which already improves upon the state of the art). Finally, Theorem~\ref{thm:main} is proven in Section~\ref{sec:main}.

\subsection*{Acknowledgments} 
We thank Pavel Valtr for his insightful comments and encouragement. OAM was supported by MIT UROP. OSP was supported in part by NSF--Simons collaboration grant DMS-2031883.

\section{A suggestive result}\label{sec:layers}

Here, we prove the theorem below, which constitutes a lower bound on the number of $5$-holes induced by $P$ in terms of the number of layers $L(P)$. During the proof of this bound, we make use of a certain assignment which associates to some elements of $P$ a collection of $5$-holes induced by $P$. Our goal in later sections will be to refine this assignment and better understand how it behaves depending on certain structural properties of $P$; this will eventually lead to a proof of Theorem~\ref{thm:main}.

\begin{theorem}\label{thm:layers}
    Every set $P$ of $n$ points in general position on the plane induces at least \[\Omega\left(\frac{n^2}{L(P)}\right)\] $5$-holes. (Recall that $L(P)$ denotes the number of layers of $P$.) Moreover, if $\ell_1$ denotes the first layer of $P$, then the number of $5$-holes induced by $P$ is also bounded from below by $\Omega(|\ell_1|\cdot n)$.
\end{theorem}

\noindent\textit{Remark.} Let us clarify once again that we make no real effort to optimize the constants hidden by the asymptotic notation throughout the paper.
\medskip

The proof of this will require the following lemma.

\begin{lemma}\label{lem:pentagon_selection}
    Let $P$ be a set of $11$ points in general position on the plane, and let $p\in P$ lie on the boundary of its convex hull. Consider the radial order $p_1,\dots,p_{10}$ of $P$ with respect to $P$. Then, at least one of the two following statements must be true:

    \begin{itemize}
        \item there exists a $5$-hole in $P$ which has $p$ as a vertex;
        \item there exists a $5$-hole in $P$ which has among its vertices two points $a$ and $b$ such that the triangle $\triangle pab$ is empty.
    \end{itemize}
\end{lemma}

\begin{proof}
    Assume that there is no $5$-hole in $P$ which has $p$ as a vertex and consider the radial order $p_1,\dots,p_{10}$ of $P$ with respect to $p$. The set $\{p,p_1,p_2,\dots,p_{9}\}$ contains $10$ points, and so it must induce at least one $5$-hole, which we denote by $H_1$. If there is some index $1\leq i\leq8$ such that $p_i$ and $p_{i+1}$ are both vertices of $H_1$, then we are done, as the triangle $\triangle pp_ip_{i+1}$ is clearly empty. Otherwise, the vertices of $H_1$ must be exactly $p_1$, $p_3$, $p_5$, $p_7$ and $p_9$. Applying the same reasoning to the set $\{p,p_2,p_3,\dots,p_{10}\}$, we conclude that $p_2$, $p_4$, $p_6$, $p_8$ and $p_{10}$ are the vertices of another $5$-hole $H_2$.
    
    Notice that if, for some index $1\leq i \leq 8$, we had that the triangle $\triangle pp_ip_{i+2}$ does not contain the point $p_{i+1}$, then such a triangle would be empty. As $p_i$ and $p_{i+2}$ both belong to either the $5$-hole $H_1$ or the $5$-hole $H_2$, we would get that at least one of these two $5$-holes fulfills the second property mentioned in the lemma. Thus, suppose that the triangle $\triangle pp_ip_{i+2}$ contains $p_{i+1}$ for all $1\leq i\leq 8$. Notice then that $p_1p_2\dots p_{10}$ is a $10$-hole in $P$. In particular, $p_1p_2p_3p_4p_5$ is a $5$-hole. As the triangle $\triangle pp_1p_2$ is empty, this concludes the proof of the lemma.
\end{proof}

We move on to studying the connection between $L(P)$ and the number of $5$-holes.
\begin{proof}[Proof of Theorem~\ref{thm:layers}]
    We may assume that $n$ is a multiple of $80$; this is done simply to avoid having to worry about writing floor functions later on. 
    
    Consider the layer decomposition $(\ell_1,\dots,\ell_{L(P)})$ of $P$. Let $k_{\operatorname{mid}}$ be the largest positive integer such that $\sum_{i=1}^{k_{\operatorname{mid}}-1}|\ell_i|< n/2$, and note that $\sum_{i=1}^{k_{\operatorname{mid}}}|\ell_i|\geq  n/2$ and $\sum_{i=k_{\operatorname{mid}}}^{L(P)}|\ell_i|> n/2$. To each element of $\ell_1\cup\dots\cup\ell_{k_{\operatorname{mid}}}$ we will assign a collection of at least $ n/80$ distinct $5$-holes in such a way that every $5$-hole is assigned to at most $40$ points of $\ell_i$, other than its vertices, for every $1\leq i\leq k_{\operatorname{mid}}$. Note that having such an assignment would immediately imply the existence of at least \[\frac{(n/2)\cdot (n/80)}{40k_{\operatorname{mid}}+5}\] distinct $5$-holes, thus proving the first part of the statement (as $k_{\operatorname{mid}}\leq L)$. Similarly, by restricting our attention to the $5$-holes assigned to vertices in $\ell_1$, we deduce that the number of $5$-holes is at least \[\frac{|\ell_1|\cdot( n/80)}{45}\,,\] implying the second part of the statement as well.
    
    Let us describe the collection of $5$-holes that will be assigned to a given point $p\in \ell_1\cup\dots\cup\ell_{k_{\operatorname{mid}}}$. Suppose that $p\in \ell_i$ and write $P_i=\ell_i\cup\dots\cup\ell_{L(P)}$. Then, $p$ is a vertex of the convex hull of $P_i$, and we can consider the radial order $p_1,\dots,p_m$ of $P_i$ with respect to $p$. Since $\sum_{i=k_{\operatorname{mid}}}^{L(P)}|\ell_i|> n/2$, $m\geq n/2$. We split the points $p_1,\dots,p_{10\cdot\lfloor m/10\rfloor}$ into disjoint blocks of size $10$: $(p_1,\dots,p_{10})$, $\dots$, $(p_{10\cdot\lfloor m/10-1\rfloor+1},\dots, p_{10\cdot\lfloor m/10\rfloor})$. We number these blocks as $B_1,\dots,B_{10\lfloor m/10\rfloor}$ (in the order they were listed). Observe that each $B_i$, together with $p$, constitutes a set of $11$ points which has $p$ as a vertex of its convex hull, and that the only elements of $P$ which are contained in this convex hull are those in $B_{i}\cup\{p\}$. For each block $B_i$ we will assign a $5$-hole to $p$, chosen as follows. If there is some $5$-hole in $P$ which has both $p$ and some element of $B_i$ among its vertices, then we assign this $5$-hole to $p$. Otherwise, by Lemma~\ref{lem:pentagon_selection}, there must exist some $5$-hole with vertices in $B_i$ such that two of its vertices $a$ and $b$ satisfy that the triangle $\triangle pab$ is empty; we assign this $5$-hole to $p$, and we say that $a$ and $b$ are the \textit{anchors} of this assignment (if there are multiple pairs of vertices of $H$ which can act as anchors, we fix one arbitrarily). Some $5$-holes might have been assigned to $p$ more than once.

    We now show that this assignment has the properties stated above. For starters, including all repeated assignments, there will be $\lfloor m/10\rfloor\geq n/20$ $5$-holes assigned to $p$. Moreover, in order for a $5$-hole to be assigned $t$ times to $p$, this $5$-hole must have $p$ as a vertex, and must also have $t$ other vertices in different blocks of the form $B_i$. It follows that no $5$-hole is assigned to $p$ more than four times, and thus there are at least $n/80$ distinct $5$-holes assigned to $p$. 
    
    Now, suppose that some $5$-hole $H$ which does not have $p$ as a vertex has been assigned to $p$. Then, $H$ must contain two vertices $a$ and $b$ such that the triangle $\triangle pab$ is empty ($a$ and $b$ are the anchors of the assignment). Recall that the elements of $\ell_i$ are the vertices of a convex polygon, and denote by $q$ and $r$ the two vertices which are adjacent to $p$ in this polygon. We claim that $\overleftrightarrow{ab}$ must intersect either $\overline{pq}$ or $\overline{pr}$. If this were not the case, then the points $p$, $a$, $b$, $q$ and $r$ would be the vertices of a convex pentagon; without loss of generality, assume that the vertices of this pentagon appear in the order $p,q,a,b,r$ along its boundary (the only other option is $p,r,a,b,q$, which is symmetric). Let $q'$ be the element of $P$ in $\triangle pqa$ with smallest non-zero Euclidean distance to $\overleftrightarrow{pa}$, and let $r'$ denote the element of $P$ in $\triangle pbr$ with smallest non-zero Euclidean distance to $\overleftrightarrow{pb}$; it could occur that $q=q'$ and/or $r=r'$. Then, since $\triangle pab$ is empty, $pq'abr'$ is a $5$-hole (see Figure~\ref{fig:5-hole-assignment}), and this $5$-hole would have been assigned to $p$ instead of $H$. This implies the claim. 
    
    The fact that $\overleftrightarrow{ab}$ intersects $\overline{pq}$ or $\overline{pr}$ tells us that $H$ can be assigned to at most 
    $4$ points from $\ell_i$ with $a$ and $b$ as anchors (the line $\overleftrightarrow{ab}$ crosses at most two side of the convex polygon formed by the elements of $\ell_i$, and each such crossing yields at most two assignments). It follows that each $5$-hole is assigned to at most $\binom{5}{2}\cdot 4=40$ points of each layer which are not its vertices, as desired.
\end{proof}

\begin{figure}[h]
    \includegraphics[width=.8\linewidth]{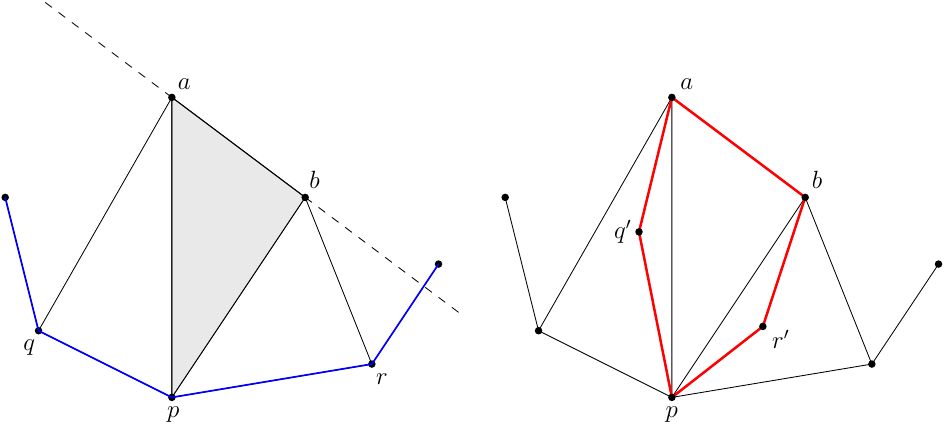}
    \caption{On the left, a portion of the layer $\ell_i$ has been highlighted in blue. In particular, we see the point $p$ and its adjacent vertices $q$ and $r$. The triangle $\triangle pab$ is empty, and we assume that the line $\overleftrightarrow{ab}$ crosses neither $\overline{pq}$ nor $\overline{pr}$. Then, $pqabr$ is a convex pentagon. On the right, we see the points $p'$ and $r'$, which have been chosen as follows. Out of the elements of $P$ contained in $\triangle pqa$ (other than $p$ and $a$), $q'$ is the one closest to $\overleftrightarrow{pa}$. Out of the elements of $P$ contained in $\triangle pbr$ (other than $p$ and $b$), $r'$ is the one closest to $\overleftrightarrow{pb}$. The points $p$, $q'$, $a$, $b$ and $r'$ are the vertices of a $5$-hole (highlighted in red).}
    \label{fig:5-hole-assignment}
\end{figure}

\section{Properties of the assignment}\label{sec:technical_lemmas}

In this section, we study some further properties satisfied by the assignment of $5$-holes to points described in the previous section, and provide a short proof of the lower bound $h_5(n)\geq\Omega(n^{4/3})$.

\begin{lemma}\label{lem:assigned_points_convex}
    Let $H$ be a $5$-hole within $P$. Then, the points of $P$ to which $H$ has been assigned can be partitioned into $O(1)$ sets, each of which is in convex position.
\end{lemma}

\begin{proof}
    It suffices to study those elements of $P$ which are not vertices of $H$ but have $H$ assigned to them (as the vertices of $H$ form a set in convex position). We divide the set of all these vertices into $20$ subsets. First, we divide into 10 subsets, according to which two vertices $a$ and $b$ of $H$ are the anchors of the assignment. 
    Then, we divide each of these $10$ subsets into two smaller subsets: the subset $P_{ab}$, which is the set of points $p$ that have $a$ and $ b$ as anchors and lie on the half-plane determined by $\overleftrightarrow{ab}$ for which $\measuredangle apb<\pi$, and the subset $P_{ab}$, which is the set of points (also with anchors $a$ and $b$) which lie on the other half-plane determined by $\overleftrightarrow{ab}$.
    
    Consider one of these $20$ subsets, say, $P_{ab}$. Let $R$ be the convex hull of $P_{ab}$. Suppose that there is a point $p\in P_{ab}$ which belongs to the interior of $R$. As the triangle $\triangle pab$ is empty and all points in $P_{ab}$ are on the same side with respect to $\overleftrightarrow{ab}$, it is clear that $\overrightarrow{pa}$ and $\overrightarrow{pb}$ must intersect the same side of the convex hull $R$ (otherwise, some vertex of $R$ would lie within $\triangle pab$). Let $q$ and $r$ be the endpoints of the side that these two rays intersect. Notice that $p$, $q$, $a$, $b$ and $r$ are the vertices of a convex pentagon and suppose, without loss of generality, that the points appear in that order around the pentagon's boundary---it is clear that the only other possible option is $prabq$. Define $q'$ (resp. $r'$) as the point in $P$ with the smallest non-zero Euclidean distance to $\overleftrightarrow{pa}$ (resp. $\overleftrightarrow{pb}$) that is contained in the triangle $\triangle pqa$ (resp. $\triangle pbr$). Then, $pq'abr'$ is a $5$-hole (see Figure~\ref{fig:convex_arc_partition}). This is not possible, as then $H$ would not have been assigned to $p$ in its respective block, because there is a $5$-hole which contains $p$ and at least one point from that block as vertices. Thus, all points in $P_{ab}$ belong to the boundary of $R$, which means that $P_{ab}$ is in convex position, as desired.
\end{proof}

\begin{figure}[h!t]
    \centering
    \includegraphics[width=.6\linewidth]{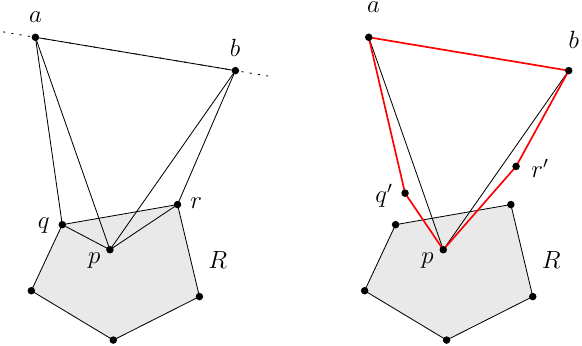}
    \caption{On the left, we see the convex hull $R$ of $P_{ab}$. If $p\in P$ is a point in the interior of $R$, then the segments $\overline{pa}$ and $\overline{pb}$ must cross the same side of $R$. It follows that $pqabr$ is a convex pentagon. As in the proof of Theorem~\ref{thm:layers}, there must exist some $5$-hole of the form $pq'abr'$.}
    \label{fig:convex_arc_partition}
\end{figure}

The tools we have developed thus far already allow us to improve upon Theorem~\ref{thm:aich}.

\begin{theorem}\label{thm:preliminary_bound}
    Every set $P$ of $n$ points in general position on the plane induces at least $\Omega(n^{4/3})$ $5$-holes.
\end{theorem}

\begin{proof}
    Our starting point is the assignment introduced during the proof of Theorem~\ref{thm:layers}. Let $M$ denote the largest integer such that there is a $5$-hole assigned to $M$ elements of $P$. Then, since there are at least $\Omega(n^2)$ assignments overall, the number of $5$-holes must be lower bounded by $\Omega(n^2/M)$. 
    
    On the other hand, Lemma~\ref{lem:assigned_points_convex} tells us that if there is a $5$-hole which has been assigned to $M$ points, then $P$ must contain a subset of at least $\Omega(M)$ points which are in convex position. Let $Q\subseteq P$ be such a subset, and let $\widetilde Q\subseteq P$ denote the subset of $P$ which consists of all those points contained in the convex hull of $Q$. Note that $Q\subseteq \widetilde Q$; in fact, $Q$ is the first layer of $\widetilde Q$. Now, Theorem~\ref{thm:layers} implies the existence of at least $\Omega(|Q|\cdot|\widetilde Q|)\geq\Omega(M^2)$ $5$-holes. Putting things together, we conclude that $P$ induces at least \[\Omega\left(\max\{n^2/M,M^2\}\right)\] $5$-holes, and the $\Omega(n^{4/3})$ lower bound follows.
\end{proof}

\subsection*{Refined assignment of 5-holes} The arguments used throughout the rest of the paper will be of similar flavor to the ones that have appeared above. However, to further improve the lower bound for $h_5(n)$, we will need to consider a modified version of the assignment of pentagons from Section~\ref{sec:layers}. Thus far, the blocks $B_1,\dots,B_{10\lfloor m/10\rfloor}$ corresponding to a point $p$ in layer $\ell_i$ ($i\leq k_{\operatorname{mid}}$) have been constructed by simply ordering all points in $P_i=\ell_i\cup\dots\cup\ell_{L(P)}$ radially around $p$ and then creating a new block every $10$ points. However, it will be convenient to consider more general families of blocks, where the blocks are pairwise-disjoint and each of them still consists of $10$ consecutive points in the radial order around $p$, but we are now allowed to have some points scattered throughout the radial order which do not belong to any of the blocks. Note that as long as the number of blocks constructed along the radial order of $p$ is at least $\Omega(n)$ for every $p$ in the first $k_{\operatorname{mid}}$ layers of $P$, all of our previous arguments apply almost verbatim. The way in which we assign a $5$-hole to $p$ for each block will also need to be made more precise, but first we require some definitions.

For a $5$-hole $H$ assigned to a point $p$ in layer $\ell_i$ which is not among the vertices of $H$, the \textit{sector} of $H$ with respect to $p$ is the totally ordered set consisting of all points of $P_i$ which either belong to $H$ or, in the radial order with respect to $p$, have at least one vertex of $H$ before them and at least one vertex of $H$ after them (see Figure~\ref{fig:sectors}a); the elements of the sector are ordered according to the radial order around $p$. Recall that every $5$-hole which is assigned to $p$ but does not have $p$ as a vertex must be contained in a block of $10$ consecutive points of $P_i$ when the set is ordered radially with respect to $p$. In particular, no sector as above can contain more than $10$ points.

Given a $5$-hole $H$ assigned to a point $p$ not in $H$, we say that this assignment is \textit{bad} if the points in the sector $Q$ of $H$ with respect to $p$ are in convex position and, furthermore, the convex hull of $Q\cup\{p\}$ is a triangle (see Figure~\ref{fig:sectors}b). We say that an assignment is \textit{good} if it is not bad (this includes the case where $p$ is a vertex of $H$).  A good assignment of a $5$-hole $H$ to some point $p$ not in $H$ will be called \textit{dubious} if the convex hull of $H\cup\{p\}$ is a triangle. A good assignment will be called \textit{trustworthy} if it is not dubious.

\begin{figure}[h]
    \centering
    \includegraphics[width=.47\linewidth]{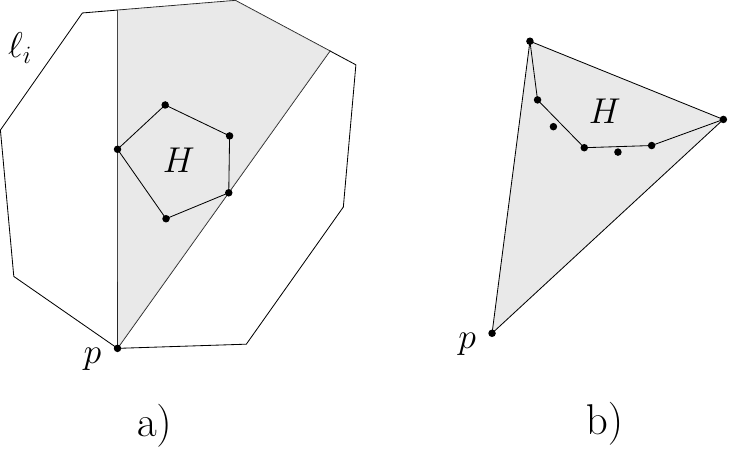}
    \caption{a) A point $p$ in layer $\ell_i$ and a $5$-hole $H$ assigned to $p$. The sector of $H$ with respect to $p$ consists of those points of $P$ which lie in the shaded region. b) The points in the sector of $H$ with respect to $p$ are in convex position. Furthermore, the convex hull of $H\cup\{p\}$ is a triangle. These are the properties that define a bad assignment.}
    \label{fig:sectors}
\end{figure}

We can now describe the refined assignment of $5$-holes, as promised. Let $p$ be a point in some layer $\ell_i$ with $i\leq k_{\operatorname{mid}}$, and let $B_1,\dots,B_t$ be some pre-selected disjoint blocks each of which consists of $10$ elements of $P_i\backslash\{i\}$ which are consecutive in the radial order of $P_i$ with respect to $p$. For each $B_i$, we assign to $p$ a $5$-hole as follows: If there is some $5$-hole which has both $p$ and some element of $B_i$ among its vertices, then we simply assign it to $p$ (as we had done earlier). However, if there is no such $5$-hole, we assign a $5$-hole to $p$ according to the rules we describe next. By Lemma~\ref{lem:pentagon_selection}, the family $F$ of $5$-holes within $B_i$ which have two vertices that together with $p$ form an empty triangle has at least one element. Out of all elements of $F$, we first try to assign to $p$ a $5$-hole which would generate a good and trustworthy assignment. If this fails, we then try to find a good but dubious assignment (still among the $5$-holes in $F$). If this is still impossible, we then chose any element of $F$ and assign it to $p$. In other words, in case there is no $5$-hole with $p$ and some element of $B_i$ among its vertices, the assignment prioritizes good assignments over bad ones, and trustworthy assignments over dubious ones. The anchors of any assignment where $p$ is not a vertex of the $5$-hole are still defined as before.

Note that we have not yet specified how the blocks $B_1,\dots,B_t$ are to be selected for any of the points in the first $k_{\operatorname{mid}}$ layers. This will also require some care, and the details are postponed until the end of this section (see Proposition~\ref{thm:concavity_yields_good_hole}). For now, however, we move on to studying some properties that the assignment satisfies regardless of the choice of blocks.

\subsection*{Studying the refined assignment}

While Lemma~\ref{lem:assigned_points_convex} already yields some powerful structural information for configurations where a single $5$-hole has been assigned to many points, it turns out that we can go even further in this direction. Pursuing this will eventually lead to claim~\ref{claim:few_good_assignments} and Proposition~\ref{thm:concavity_yields_good_hole}---the main results of this section---which respectively tell us that no $5$-hole can be involved in more than $O(1)$ assignments, and that it is possible to produce many good assignments of $5$-holes, unless the structure of $P$ is "simple". We begin with some preparatory results.

\begin{proposition}\label{prop:first_structure}
For every positive integer $k$ there exists some other positive integer $K_1=K_1(k)$ such that, among any $K_1$ points to which $H$ has been assigned, we can find a subset $S$ of $k$ points for which the following properties are satisfied:
\begin{itemize}
    \item all elements of $S$ are assigned to $H$ via the same anchors $a$ and $b$ and, for any two distinct vertices $a'$ and $b'$ of $H$, all elements of $S$ lie on the same side of $\overleftrightarrow{a'b'}$;

    \item the elements of $S$ are in convex position;

    \item there is a labeling $s_1,\dots,s_{k}$ of the elements of $S$ so that $s_1\dots s_k$ is a convex polygon (in that order) and, for $1\leq i<j\leq k$, every vertex $h$ of $H$ satisfies $\measuredangle s_is_jh<\pi$.
\end{itemize}
\end{proposition}

See Figure~\ref{fig:forbidden_configuration}a for an example where the three properties in the statement are satisfied.

\begin{proof}
If there are $K_1$ points as in the statement of the lemma, then we can choose at least $(K_1-5)/20$ of them which are all assigned to $H$ via the same anchors $a$ and $b$, and which lie on the same side of $\overleftrightarrow{ab}$. Given that any $9$ lines split the plane into at most $46$ regions, we can pass to a further subset $\tilde S$ of at least $(K_1-5)/(20\cdot 46)$ points, all of which lie on the same side of $\overleftrightarrow{a'b'}$ for any two vertices $a'$ and $b'$ of $H$.  As we saw in the proof of Lemma~\ref{lem:assigned_points_convex}, the elements of $\tilde S$ must be in convex position, and thus $\tilde S$ satisfies the first two required properties. We will select a subset $S\subseteq\tilde S$ of size at least $|\tilde S|/6$ which also satisfies the third required property.

We construct an auxiliary directed graph $G_{\tilde S}$ with vertex set $\tilde S$ and a directed edge from $s$ to $s'$ if and only if the ray $\overrightarrow{ss'}$ intersects the interior of $H$. Suppose that $s$ and $s'$ are such that there is a directed edge from $s$ to $s'$, and that $s\in \ell_{i}$. The point $s'$ must belong to the convex hull of $H\cup\{s\}$, or else $s'$ would lie on the opposite side of $s$ with respect to at least one of the lines determined by two vertices of $H$. In particular, this tells us that $s'\in P_i$. Since all vertices of $H$ belong to a block of $10$ consecutive points in the radial order of $P_i$ with respect to $s$ and no element of $\tilde S$ is a vertex of $H$, there are at most five elements of $\tilde S$ within $H\cup\{s\}$, other than $s$ itself. Hence, every vertex of $G_{\tilde S}$ has at most $5$ out-going edges. A straightforward inductive argument (or, alternatively, a greedy algorithm) can now be used to show that there exists a subset $S\subseteq \tilde S$ of size at least $|\tilde S|/6$ such that there is no directed edge joining two elements of $S$. Overall, we have that $|S|\geq \Omega(K_1)$.

Recalling that $S$ is in convex position, it is not hard to see that the only way $S$ could fail to satisfy the third property in the statement is if there is some point $s\in S$ such that the convex hull of $H\cup\{s\}$ intersects the interior of the convex hull $R$ of $S$. Let $h_1$ and $h_2$ be the first and last points in the radial order of $H\cup\{s\}$ with respect to $s$ ($h_1$ and $h_2$ are vertices of $H$). That the convex hull of $H\cup\{s\}$ intersects the interior of $R$ implies that at least one of the segments $\overline{sh_1}$ and $\overline{sh_2}$ intersects the interior of one of the sides of $R$. In fact, by our choice of $S$, $s$ is the only element of $S$ within the convex hull of $H\cup\{s\}$; this is only possible if both of these segments actually cross the same side $\overline{qr}$ of $R$. In this configuration, we get that $sqabr$ is a convex pentagon (it could also happen that the convex pentagon is $srabq$, but this case is symmetric). See Figure~\ref{fig:forbidden_configuration}b. As in the proof Lemma~\ref{fig:convex_arc_partition}, we can now find a $5$-hole of the form $sq'abr'$, which should have been assigned to $s$ instead of $H$. See Figure~\ref{fig:forbidden_configuration}c. This contradiction shows that $S$ satisfies also the third property in the statement, as required.
\end{proof}

\begin{figure}[h]
    \centering
    \includegraphics[width=.725\linewidth]{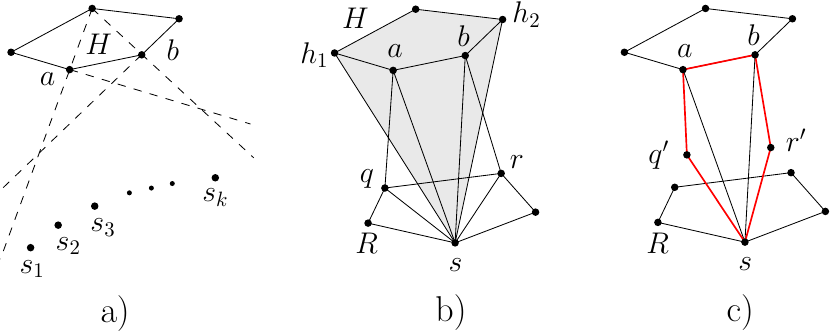}
    \caption{a) An example where the elements of $S$ satisfy all properties in the statement of Proposition~\ref{prop:first_structure}.
    b) The convex hull of $H\cup\{s\}$ has been shaded. The segments $\overline{sh_1}$ and $\overline{sh_2}$ both cross the same side $\overline{qr}$ of $R$, and thus $sqabr$ is a convex pentagon. c) We can find a $5$-hole of the form $sq'abr'$ (shown in red). It could occur that $q'=q$ and/or $r=r'$.}
    \label{fig:forbidden_configuration}
\end{figure}

        By further strengthening Proposition~\ref{prop:first_structure}, we obtain the following.

\begin{proposition}\label{prop:second_structure}
For every positive integer $k$ there exists some positive integer $K_2=K_2(k)$ such that, among any $K_2$ points to which $H$ has been assigned, we can find a subset $S$ of at least $k$ points for which: 
\begin{itemize}
    \item the three properties in the statement of Proposition~\ref{prop:first_structure} are satisfied;

    \item the sectors of $H$ with respect to every element of $S$ are all identical (including the order).
\end{itemize}
\end{proposition}

\begin{proof} By making $K_2$ sufficiently large, we may assume that there is a set $\tilde{S}$ of $10^{10}k$ points for which all three properties in the statement of Proposition~\ref{prop:first_structure} hold. Let $a$ and $b$ be the anchors of all assignments of $H$ to the elements of $\tilde S$, and label the elements of $\tilde S$ as $s_1,\dots,s_{10^{10}k}$ so that the third property in the statement of Proposition~\ref{prop:first_structure} holds. Without loss of generality, we assume that $\measuredangle as_ib<\pi$ for every $i$. 

For every $1\leq i\leq 10^{10}k$, each element $q\not\in H$ of the sector of $H$ with respect to $s_i$ will be classified as being either \textit{nearby} or \textit{away} with respect to $s_i$, depending on whether $q$ is contained in the convex hull of $H\cup\{p\}$ or not, respectively. By the first property in the statement of Proposition~\ref{prop:first_structure}, if some point $q\notin H$ is in the sector of $H$ with respect to two points $s_i$ and $s_j$, then it is either nearby with respect to both or away with respect to both. Since $\triangle s_iab$ is empty for every $s_i$, every point $q\notin H$ which is nearby with respect to $s_i$ can be seen to satisfy either $\measuredangle qs_ia<\pi$ or $\measuredangle bs_iq<\pi$ (but not both).

Suppose that the point $q\notin H$ is nearby with respect to some $s_i$ with $i>2$, and that $\measuredangle bs_iq<\pi$. Consider any $s_j$ with $j>i$. We claim that $q$ is nearby with respect to $s_j$, and that $\measuredangle bs_jq<\pi$. Notice that $q$ belongs to the convex hull of $H\cup\{s_i\}$. Thus, by the third property in the statement of Proposition~\ref{prop:first_structure}, it must be the case that $q\in\angle bs_is_j$. For the sake of contradiction, suppose that $\measuredangle qs_jb<\pi$ (or, equivalently, that $\measuredangle bs_jq>\pi$). Since $\triangle s_jab$ is empty, this can only occur if $q\in\angle s_is_ja$ (see Figure~\ref{fig:s1sisj}). As $q$ is in the sector of $H$ with respect to $s_i$, we can use the third property from Proposition~\ref{prop:first_structure} with $s_1$ and $s_i$ to conclude that $\measuredangle s_1s_jq<\pi$ and thus $s_1abqs_i$ is a convex pentagon (again, see Figure~\ref{fig:s1sisj}). Let $s_1'$ (resp. $q'$) denote the element of $P$ with the smallest non-zero Euclidean distance to $\overleftrightarrow{s_ia}$ (resp. $\overleftrightarrow{s_ib}$) that is contained in the triangle $\triangle s_1as_i$ (resp. $\triangle s_ibq$). Then, $s_1'abq's_i$ is a $5$-hole. This, however, contradicts our initial assumption that $H$ is assigned to $s_i$, because this $5$-hole would have been assigned to $s_i$ instead of $H$. Thus, $\measuredangle bs_jq<\pi$ must hold. It is not hard to see that $q$ must then belong to the sector of $H$ with respect to $s_j$, and thus must also be nearby with respect to this point. 

\begin{figure}[h]
    \centering
    \includegraphics[width=.42\linewidth]{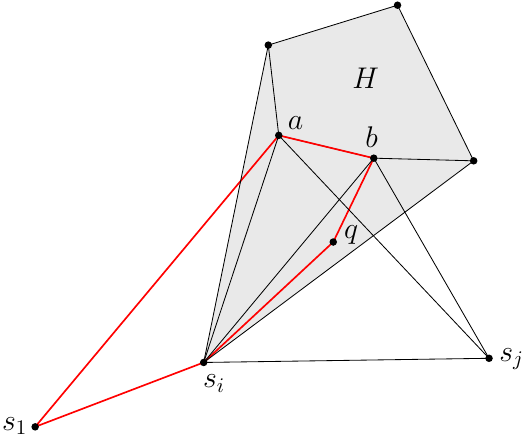}
    \caption{The convex hull of $H\cup\{s_i\}$ has been shaded. The point $q$ must belong to the angular regions $\angle bs_is_j$ and $\angle s_is_ja$. Moreover, since $q$ is in the shaded region, the third property in Proposition~\ref{prop:first_structure} implies that $\measuredangle s_1s_is_j<\pi$. It is then easy to see that $s_1abqs_i$ (shown in red) is a convex pentagon.}
    \label{fig:s1sisj}
\end{figure}

By a symmetric argument, if $q\not\in H$ is nearby with respect to some $s_i$ with $i<10^{10}k$ and $\measuredangle qs_ia<\pi$, then $q$ is also nearby with respect to every $s_j$ with $j<i$, and it must be the case that $\measuredangle qs_ja<\pi$.

In particular, the above discussion implies that if $q\not\in H$ is nearby with respect to some $s_i$ with $1<i<10^{10}k$, then $q$ must be nearby with respect to $s_1$ or $s_{10^{10}k}$ (possibly both). Given that there are at most $5$ points which are nearby with respect to $s_1$ and at most $5$ points which are nearby with respect to $s_{10^{10}k}$, there is a set of at most $10$ points which contains all points $q$ that are nearby with respect to at least one element of $\tilde S$. For each $s_i$, given the sector $Q$ of $H$ with respect to $s_i$, we define the \textit{nearby sector} of $H$ with respect to $s_i$ as the ordered set obtained from $Q$ by ignoring all the away points. By the Pigeonhole-Principle, there is a subset $ S$ of $\tilde S$ of size at least $10^{10}k/(2^{10}10!)>k$ for which the corresponding nearby sectors of $H$ are all identical; let $p_1,\dots,p_t$ be the elements of these nearby sectors, in order. Note that $5\leq t\leq 10$, as the vertices of $H$ must lie among these points.

Next, we show that there is no index $1<i<t$ with $\measuredangle p_{i-1}p_ip_{i+1}>\pi$. Indeed, suppose that $\measuredangle p_{i-1}p_ip_{i+1}>\pi$ and let $s$ and $s'$ be two distinct elements of $S$. Then, without loss of generality, $sp_{i-1}p_ip_{i+1}s'$ is a convex pentagon, in that order. Furthermore, the convex quadrilaterals $sp_{i-1}p_ip_{i+1}$ and $s'p_{i-1}p_ip_{i+1}$ are both empty. Hence, if $s''\in P$ is the point in $\triangle sp_{i+1}s'$ with smallest non-zero Euclidean distance to $\overleftrightarrow{sp_{i+1}}$, then $sp_{i-1}p_ip_{i+1}s''$ is a $5$-hole which would have been assigned to $s$ instead of $H$ (see Figure~\ref{fig:pushing_from_top}a). This is a contradiction. Now, since $\measuredangle p_{i-1}p_ip_{i+1}>\pi$ for all $1<i<t$, the assignment of $H$ to $s_i$ must be either bad or dubious for every $s_i\in S$.

Finally, we claim that there are no away points in the sector of $H$ with respect to any of the elements of $S$, and thus all these sectors must be outright identical. Suppose, for the sake of contradiction, that there exists some point $q$ that is away with respect to some $s\in S$. Then, there exists some index $1\leq j\leq t-3$ with $q\in\angle p_jsp_{j+3}$, and it must be the case that $\measuredangle p_jqp_{j+3}>\pi$. Let $q'\in P$ be the point in $\triangle p_jqp_{j+3}$ with smallest non-zero Euclidean distance to $\overleftrightarrow{p_jp_{j+3}}$. Then, $p_jp_{j+1}p_{j+2}p_{j+3}q'$ is a $5$-hole (see Figure~\ref{fig:pushing_from_top}b). However, the assignment of this $5$-hole to $s$ would be good and trustworthy, while the assignment of $H$ to $s$ is dubious. This contradiction concludes the proof.
\end{proof} 
\begin{figure}[h]
    \centering
    \includegraphics[width=.51\linewidth]{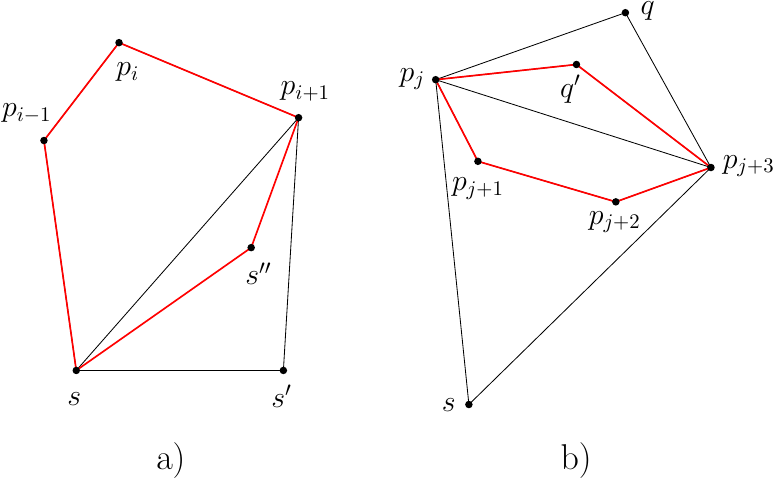}
    \caption{a) The quadrilateral $sp_{i-1}p_ip_{i+1}$ is a $4$-hole, and $sp_{i-1}p_ip_{i+1}s'$ is a convex pentagon. Thus, if $s''$ is the point in $\triangle sp_{i+1}s'$ with smallest non-zero Euclidean distance to $\overleftrightarrow{sp_{i+1}}$, then $sp_{i-1}p_ip_{i+1}s''$ is a $5$-hole. b) The point $q$ lies in $\angle p_{j}sp_{j+3}$. The point $q'$ has the smallest non-zero Euclidean distance to $\overleftrightarrow{p_jp_{j+3}}$ among the points in $\triangle p_jqp_{j+3}$. The pentagon $p_jp_{j+1}p_{j+2}p_{j+3}q'$ (shown in red) is a $5$-hole which induces a trustworthy assignment.}
    \label{fig:pushing_from_top}
\end{figure}

Pushing the ideas in the above proof a little further, we get the following crucial fact.

\begin{claim}\label{claim:few_good_assignments}
    Every $5$-hole is involved in a total of at most $O(1)$ good assignments.
\end{claim}

\begin{proof}
    Suppose that some $5$-hole $H$ is involved in $K$ assignments for some positive integer $K$. If $K$ is sufficiently large, then there are two of these $K$ points, say $s_1$ and $s_2$, for which the conditions in Proposition~\ref{prop:second_structure} are satisfied. Notice that, as these points are assigned to $H$ through a good assignment and because the sectors of $H$ with respect to $s_1$ and $s_2$ are the same, there must be three consecutive points (in the radial order with respect to both $s_1$ and $s_2$) in that sector, $a$, $b$ and $c$ (in that order), so that $\measuredangle abc > \pi$. 
    
    Notice that the quadrilaterals $abcs_1$ and $abcs_2$ are $4$-holes. Furthermore, without loss of generality, the pentagon $s_1abcs_2$ is convex. Thus, if $s_2'$ denotes the element of $P$ in $\triangle s_1cs_2$ with smallest non-zero Euclidean distance to $\overleftrightarrow{s_1c}$, the pentagon $s_1abcs_2'$ must be a $5$-hole (this is essentially the same situation that was depicted in Figure~\ref{fig:pushing_from_top}a). This is a contradiction, as $H$ would have not been assigned to $s_1$ if there is a $5$-hole which has $s_1$ and at least one point from the sector of $H$ with respect to $s_1$ among its vertices---$5$-holes of this kind have a higher priority in our assignment process.
\end{proof}

\subsection*{Finding many good assignments} Our last goal in this section will be to show that, under mild assumptions, one can produce many good assignments for a given point in one of the first $k_{\operatorname{mid}}$ layers (Proposition~\ref{thm:concavity_yields_good_hole}).
The key tool which will allow us to achieve this is the following geometric lemma.

\begin{lemma}\label{lem:concavity_yields_good}
    Let $p$ be a point in layer $\ell_i$ and suppose that $p_1,\dots,p_{11}$ are $11$ consecutive points in the radial order of $P_i$ with respect to $p$. If $\measuredangle p_{5}p_6p_7>\pi$, then we can choose a block $B$ of $10$ consecutive points among these $11$ such that assigning a $5$-hole to $p$ using $B$ results in a good assignment.
\end{lemma}

\begin{proof}
    For $i\in \{3, 4 ... 9\}$, we say that $i$ is \textit{sharp} if $\measuredangle p_{i-1}p_i p_{i+1} > \pi$. We are given that $6$ is sharp. 
    
    Suppose first that $6$ is the only sharp number and that $\measuredangle p_4 p_5 p_7 < \pi$. Then, the pentagon $p_2p_3p_4p_5p_7$ is convex and the only way it could fail to be empty is if $p_6$ lies inside it. In fact, if $p_6$ is not inside $p_2p_3p_4p_5p_7$, not only is this pentagon a $5$-hole, but it also induces a good assignment, as $p_5$, $p_6$ and $p_7$ are inside the sector of $p_2p_3p_4p_5p_7$ with respect to $p$. If $p_6$ is inside $p_2p_3p_4p_5p_7$, then $\measuredangle p_2p_6p_7 < \pi$, and thus the pentagon $p_2p_6p_7p_8p_9$ is a $5$-hole which also induces a good assignment with respect to $p$ (see Figure~\ref{fig:technical_lemma}a). This concludes the analysis of the case where $\measuredangle p_4 p_5 p_7 < \pi$. Suppose now that $6$ is the only sharp number but $\measuredangle p_4p_5p_7 > \pi$. By symmetry, we can assume that $p_5p_7p_8 > \pi$ too. Notice now that $pp_4p_5p_7p_8$ is a $5$-hole which, naturally, constitutes good assignment with respect to $p$ (see Figure~\ref{fig:technical_lemma}b). This completes the proof for the case where only $6$ is sharp. 
    
    Assume now that there is at least one other sharp number and, without loss of generality, that it is contained in $\{3, 4, 5\}$. Consider any $5$-hole $H$ with vertices in $\{p,p_1,\dots,p_9\}$ (recall that least one such $5$-hole exists, because $h_5(10)=1$). If $H$ contains $p$ we get a good assignment, so we assume that it does not. Suppose $H$ is the pentagon $p_1p_3p_5p_7p_9$. As $\measuredangle p_5p_6p_7 > \pi$, we have that $\triangle pp_5p_7$ is empty. Hence, we can assign $H$ to $p$; this would be a good assignment, since $p_5$, $p_6$ and $p_7$ are within the sector of $H$ with respect to $p$. Suppose then that this is not the $5$-hole $H$. 
    
    There will be some index $1\leq i\leq 8$ such that $p_i$ and $p_{i+1}$ are vertices of the $5$-hole $H$. As $\triangle pp_ip_{i+1}$ is empty, $H$ is a valid assignment with respect to $p$. Furthermore, at least one of the vertices of $H$ is in $\{p_1, p_2, p_3, p_4, p_5\}$. Thus, if at least one vertex of $H$ is in $\{p_7, p_8, p_9\}$, then the $3$ points $p_5$, $p_6$, $p_7$ would belong to the sector of $H$ with respect to $p$, and this would be a good assignment. Suppose then that the vertices of $H$ are contained in $\{p_1,\dots,p_6\}$. If $5$ were sharp, then $pp_4p_5p_6p_7$ would be a $5$-hole which constitutes a good assignment with respect to $p$, so we assume that one of $\{3,4\}$ is sharp. The leftmost vertex of $H$ is either $p_1$ or $p_2$, and its rightmost vertex is either $p_5$ or $p_6$. In any case, all points $p_2$, $p_3$, $p_4$, $p_5$ are in the angular region of $H$ with respect to $p$, so $H$ yields a good assignment with respect to $p$.
    
    In summary, we can always find a good assignment under the lemma's assumptions.
\end{proof}

\begin{figure}[h]
    \centering
    \includegraphics[width=.75\linewidth]{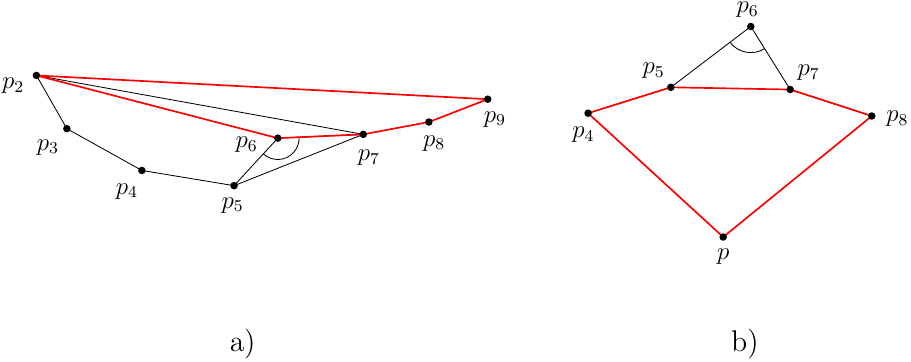}
    \caption{a) In the case where $6$ is the only sharp number and $\measuredangle p_4 p_5 p_7 < \pi$, if $p_6$ lies within the convex pentagon $p_2p_3p_4p_5p_7$, then $p_2p_6p_7p_8p_9$ (red) is a $5$-hole which induces a good assignment. b)  If $6$ is the only sharp number but $\measuredangle p_4 p_5 p_7 > \pi$ and $\measuredangle p_5p_7p_8>\pi$, we have that $pp_4p_5p_7p_8$ (red) is a $5$-hole. } 
    \label{fig:technical_lemma}
\end{figure}

As a straightforward corollary of the above lemma, we have the following.

\begin{proposition}\label{thm:concavity_yields_good_hole}
    Let $p$ be a point in layer $\ell_i$ for some $i\leq k_{\operatorname{mid}}$ and consider the radial order $p_1,\dots,p_m$ of all other elements of $P_i$ with respect to $p$. Then, we can choose $\Omega(m)$ pairwise-disjoint blocks (each consisting of $10$ consecutive points from $p_1,\dots,p_m$) such that, after carrying out the assignment of $5$-holes, the number of good assignments involving $p$ is as large as possible among all such collections of disjoint blocks. Moreover, if $r$ and $t$ satisfy $1\leq r\leq t-10$, $t\leq m$ and none of the points $p_r,p_{r+1},\dots,p_{t}$ are contained in a block which induces a good assignment, then $\measuredangle p_{j-1}p_{j}p_{j+1}<\pi$ for all $j$ with $r+5\leq j\leq t-5$. 
\end{proposition}

\begin{proof}
We start by selecting a collection of pairwise-disjoint blocks of size $10$ which induce only good assignments and is as large as possible with this property. By Lemma~\ref{lem:concavity_yields_good}, for every $6\leq j\leq m-5$ such that $\measuredangle p_{j-1}p_jp_{j+1}>\pi$, at least one of the points in $\{p_{j-5},\dots,p_{j+5}\}$ must be contained in some block of this collection. Indeed, we could otherwise add at least one more block which is disjoint from the rest and induces a good assignment. This directly yields the property stated in the last sentence of the lemma.

To conclude, for every maximal string of consecutive points $p_{r},p_{r+1},\dots,p_{t}$ which do not belong to any of the blocks in the above family, we create $\lfloor (t-r+1)/10\rfloor$ pairwise-disjoint blocks consisting of $10$ consecutive points each, and add them to our initial collection of blocks. This clearly ensures that the total number of blocks is at least $\Omega(m)$, while preserving the number of good assignments.
\end{proof}

\section{Proof of the main theorem}\label{sec:main}
For every point $p$ in one of the first $k_{\operatorname{mid}}$ layers of $P$, we denote by $g_{p}$ the number of good assignments produced by applying Proposition~\ref{thm:concavity_yields_good_hole} to $p$. Then, applying the said proposition to all such points, it is possible to simultaneously produce
\begin{equation}\label{eq:1}
\sum_{p\in \ell_1\cup\dots\cup\ell_{k_{\operatorname{mid}}}}g_p\,
\end{equation} 
good assignments of $5$-holes to points. By Claim~\ref{claim:few_good_assignments}, up to a multiplicative constant, the above expression acts also as a lower bound for the number of $5$-holes induced by $P$. We can assume from now on that there are at least $\Omega(n)$ point $p$ in the first $k_{\operatorname{mid}}$ layers for which $g_p=o(n)$, or else we would immediately deduce that there are at least $\Omega(n^2)$ $5$-holes. 

On the other hand, in the setting of Proposition~\ref{thm:concavity_yields_good_hole}, the points $p_1,\dots,p_m$ can be partitioned into at most (say) $100g_p$ sets each of which consists of consecutive points---in the radial order---and is the vertex set of an empty convex polygon. Every $5$-hole all of whose vertices belong to the same convex polygon of this kind will be said to be \textit{visible from} $p$. A $t$-sided empty convex polygon gives rise to $\binom{t}{5}$ $5$-holes, so $P$ induces at least \begin{equation}\label{eq:2}100g_p\binom{\lfloor m/(100g_p)\rfloor}{5}\,\end{equation} $5$-holes which are visible from $p$. The above expression is of order $\Omega(m^5/g_p^4)=\Omega(n^5/g_p^4)$ as long as $m/(100g_p)\geq 5$. Repeating this argument over all points in the first $k$ layers (and using the fact that $g_p=o(n)$ for some $p$), we conclude that there are at least
\begin{equation}\label{eq:3}
\Omega\left(\max_{p\in \ell_1\cup\dots\cup\ell_{k_{\operatorname{mid}}}} \frac{n^5}{g_p^4}\right)
\end{equation} 
$5$-holes in $P$. 

These two lower bounds (i.e.,~\eqref{eq:1} and~\eqref{eq:3}) together can easily be seen to imply that the number of $5$-holes determined by $P$ is at least $\Omega(n^{9/5})$. Note that the second of the two lower bounds seems quite wasteful however, as we are simply taking a maximum and ignoring the possibility that many of the $5$-holes found as $p$ ranges over $\ell_1\cup\dots\cup\ell_{k_{\operatorname{mid}}}$ are different. In order to prove Theorem~\ref{thm:main}, we will improve upon this part of our argument by using one last geometric lemma.

A $5$-hole $H$ that is visible from $p$ will be called \textit{long} with respect to $p$ if, in the radial order with respect to $p$, any two vertices of $H$ are at distance at least $10$ from each other (by which we mean that there are at least $9$ other points between them). Note that, reasoning as we did when deriving $\eqref{eq:2}$, one concludes that, as long as $m/(100g_p)> 40$, the number of $5$-holes which are visible and long with respect to $p$ is at least $\Omega(n^{5}/g_{p}^4)$. Hence, the number of pairs $(H,p)$ where $H$ is a $5$-hole and $p\in \ell_1\cup\dots\cup\ell_{k_{\operatorname{mid}}}$ is such that $H$ is visible and long with respect to the point $p$ must be at least 
\begin{equation}\label{eq:4}
    \Omega\left(\sum_{p\in \ell_1\cup\dots\cup\ell_{k_{\operatorname{mid}}}} \frac{n^5}{g_p^4}\right)\,.
\end{equation}
Here, we are implicitly using the fact that $g_p=o(n)$ for linearly many points $p$.

\begin{lemma}\label{lem:visible_implies_convex}
 For every $5$-hole $H$, the set of vertices with respect to which $H$ is visible and long can be divided into at most $5$ subsets all of which are in convex position.
\end{lemma}

\begin{proof}
    Let $S$ be the set of points with respect to which $H$ is visible and long. Any $10$ lines split the plane into at most $56$ regions, so we can partition $H$ into at most $56$ subsets so that, for every two distinct vertices $a$ and $b$ of $H$, any two elements in the same subset of the partition must lie on the same side of $\overleftrightarrow{ab}$. We show that each of these subsets of $S$ is in convex position. 

    Let $S'$ be one of these subset of $S$. For starters, any two elements of $S'$ see the vertices of $H$ in the same order, say, $a,b,c,d,e$. Pick some $s\in S'$ and suppose that $s\in \ell_i$. Let $p_1,\dots,p_r$ be the minimal string of consecutive elements in the radial order of $P_i$ with respect to $s$ which contains all vertices of $H$. Note, in particular, that $p_1=a$ and $p_r=e$; choose also $w,z\in\mathbb \{1,\dots,r\}$ such that $b=p_w$ and $d=p_z$. Since $H$ is long with respect to $s$, it must be the case that $w\geq 11$, $r-z\geq 10$ and $z-w\geq 20$. 
    
    Note that $p_1\dots p_r$ is a convex polygon, in that order, due to the how visible $5$-holes are defined. Moreover, since $H$ is visible from $s$, $\measuredangle sba<\pi$ and $\measuredangle eds<\pi$. Consider some $s'\in S$ and suppose that $s'\in \ell_j$. Given that every element of $S'$ lies on the same side as $s$ with respect to every line determined by two vertices of $H$, we must have that $\measuredangle s'ba<\pi$ and $\measuredangle eds'<\pi$ as well. This forces $b=p_w,p_{w+1},\dots,p_z=d$ to appear in this order in the radial order of $P_j$ with respect to $s'$ too. Furthermore, in this radial order, those vertices which belong to the interval between $p_w=b$ and $p_z=d$ (inclusive) are in convex position, since $H$ is also visible from $s'$. It is easy to see that this is possible only if $b=p_w,p_{w+1},\dots,p_z=d$ are consecutive in the radial order around $s'$. See Figure~\ref{fig:long_holes}a. As this holds for every $s'\in S'$, in particular, we get that no element of $\{p_w,p_{w+1}\dots,p_z\}$ is in $S'$.

    If there were two points $s\in S'$ and $p\in P$ for which $p$ is in the interior of $\triangle sp_{w}p_{w+1}$, $p$ would come after $p_w$ but before $p_{w+1}$ in the radial order with respect to $s$; this is impossible, as it contradicts the conclusion we reached in the previous paragraph. Suppose now that $S'$ is not in convex position. Then, there exists some point $s\in S'$ which belongs to the interior of the convex hull $R$ of $S'$. The proof now basically follows the proof of Lemma~\ref{lem:assigned_points_convex}. Indeed, the segments $\overline{sp_w}$ and $\overline{sp_{w+1}}$ must cross the same side of $R$, or else there would be some $s'\in S'$ in the interior of $\triangle sp_wp_{w+1}$. Let $\overline{s_1s_2}$ be the side of $R$ that is crossed by these two segments. Then, $ss_1p_wp_{w+1}s_2$ is a convex pentagon, in that order, and $\triangle sp_wp_{w+1}$ is empty. By our usual trick, we can find some $5$-hole of the form $ss_1'p_wp_{w+1}s_2'$ in $P$. See Figure~\ref{fig:long_holes}b. We could have added the block $B=\{p_w,p_{w+1},\dots,p_{w+9}\}$ to the collection of blocks corresponding to $z$. This block would induce a good assignment of a $5$-hole to $s$. Moreover, since $w\geq 11$, $r-z\geq 10$ and $z-w\geq 20$, the addition of this block would not interfere with any of the other blocks inducing good assignments for $s$. This is a contradiction, since we assumed that the construction of the blocks was such that the number of good assignments for $s$ was as large as possible. Hence, $S'$ is in convex position, and we are done.
\end{proof}

\begin{figure}[h]
    \centering
    \includegraphics[width=.63\linewidth]{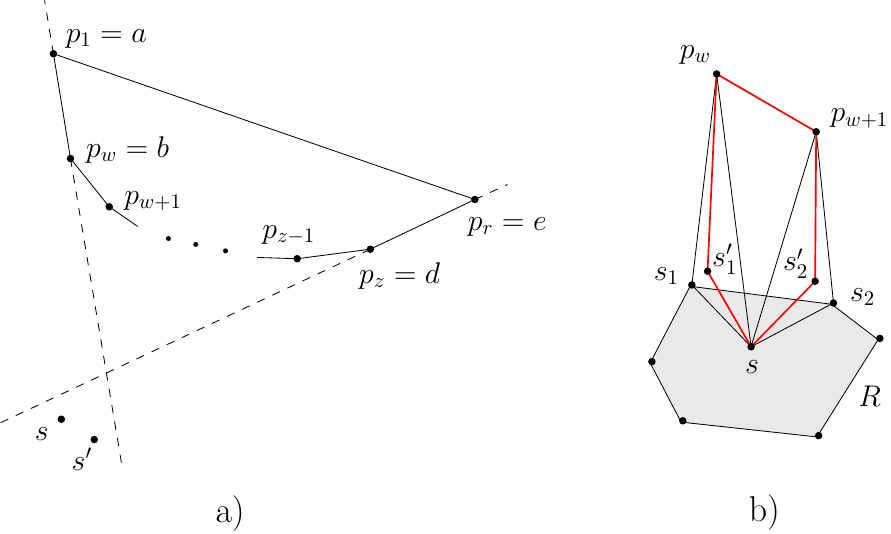}
    \caption{a) The points $p_1,p_w,p_{w+1},\dots,p_{z-1},p_{z},p_r$ appear in this order around $s$, and are in convex position. The point $s'\in \ell_j$ lies on the same side of $\overleftrightarrow{ab}$ and $\overleftrightarrow{de}$ as $s$, and thus must also see the above points in the same order. Using that $H$ is visible from $s'$, we can conclude that $p_w,p_{w+1},\dots,p_{z-1},p_{z}$ are in fact consecutive in the radial order of $P_j$ with respect to $s'$. b) The convex hull $R$ of $S'$ has been shaded. The segments $\overline{sp_w}$ and $\overline{sp_{w+1}}$ both intersect the same side $\overline{s_1s_2}$ of $R$. The triangle $\triangle sp_{w}p_{w+1}$ is empty, so we can find a $5$-hole of the form $ss_1'p_wp_{w+1}s_2'$ (shown in red). As always, it could happen that $s_1=s_1
    $ and/or $s_2=s_2'$.} 
    \label{fig:long_holes}
\end{figure}

\begin{proof}[Proof of Theorem\ref{thm:main}]
    If there is some $5$-hole which is visible and long with respect to at least $n^{10/11}$ points, then Lemma~\ref{lem:visible_implies_convex} tells us that there must be a $\Omega(n^{10/11})$-sided convex polygon with vertices in $P$. The second part of Theorem~\ref{thm:layers} would then imply the desired bound (as occurred during the proof of Theorem~\ref{thm:preliminary_bound}). Assume, thus, that every $5$-hole is visible and long from less than $n^{10/11}$ points.

    If $\sum_{p\in \ell_1\cup\dots\cup\ell_{k_{\operatorname{mid}}}} g_p\geq n^{20/11}$, then, by~\eqref{eq:1}, $P$ induces at least $\Omega(n^{20/11})$ $5$-holes. Hence, we can also assume that $\sum_{p\in \ell_1\cup\dots\cup\ell_{k_{\operatorname{mid}}}} g_p< n^{20/11}$. 

    As we noted above, the expression in~\eqref{eq:4} lower bounds the number of pairs $(p,H)$ with $H$ visible and long with respect to $p$. Applying Jensen's inequality together with $\sum_{p\in \ell_1\cup\dots\cup\ell_{k_{\operatorname{mid}}}} g_p\leq n^{20/11}$, we conclude that the number of such pairs is in turn lower bounded by \[\Omega\left(n\cdot\frac{n^5}{(n^{20/11}\cdot n^{-1})^{4}}\right)=\Omega(n^{30/11})\,.\] Lastly, by our first assumption, every $5$-hole belongs to less than $n^{10/11}$ pairs of this kind, so the number of distinct $5$-holes induced by $P$ must be at least $\Omega(n^{20/11})$, as promised.
\end{proof}

\bibliographystyle{plain}
\bibliography{refs}

\end{document}